\documentclass[11pt,reqno]{amsart}
\usepackage{latexsym,amssymb,amsmath,amsthm,hyperref}

\usepackage{enumerate}
\usepackage{comment}
\usepackage{amssymb,amsmath,amscd,amsthm,color}
\usepackage{mathrsfs}
\usepackage{mathtools}

  \newcommand{\inn}[2]{\langle #1,#2\rangle}         

  \def\tr{\operatorname{Tr}}
  \def\inv{^{-1}}
  \def\Aut{\operatorname{Aut}}
  \newcommand{\cl}{\operatorname{cl}}

  \def\clsp{\operatorname {\overline{span}}}

\newcommand{\B}{\mathcal{B}}

\newcommand{\CC}{\mathbb{C}}

\newcommand{\FF}{\mathbb{F}}
\newcommand{\F}{\mathcal{F}}

\newcommand{\K}{\mathcal{K}}

\renewcommand{\l}{\ell}

\newcommand{\NN}{\mathbb{N}}

\renewcommand{\O}{\mathcal{O}}

\newcommand{\RR}{\mathbb{R}}

\renewcommand{\S}{\mathcal{S}}
\newcommand{\TT}{\mathbb{T}}

\newcommand{\T}{\mathcal{T}}

\newcommand{\ZZ}{\mathbb{Z}}

\newcommand{\e}{\varepsilon}

\newtheorem{thm}{Theorem}[section]
\newtheorem{prop}[thm]{Proposition}
\newtheorem{lem}[thm]{Lemma}
\newtheorem{cor}[thm]{Corollary}

\theoremstyle{definition}
\newtheorem{definition}[thm]{Definition}

\theoremstyle{remark}
\newtheorem{remark}[thm]{Remark}

\newcommand{\thmref}[1]{Theorem~\ref{#1}}

\newcommand{\proref}[1]{Proposition~\ref{#1}}
\newcommand{\lemref}[1]{Lemma~\ref{#1}}
\newcommand{\corref}[1]{Corollary~\ref{#1}}

\numberwithin{equation}{section}

\begin{document}

\title[Equilibrium states and growth of quasi-lattice ordered monoids]{Equilibrium states and growth \\ of quasi-lattice ordered monoids}

\date{21 September 2017}

\author[Bruce]{Chris Bruce}
\address[Chris Bruce]{
Department of Mathematics and Statistics\\
University of Victoria\\
Victoria, BC V8W 3R4\\
Canada}%
\email[Bruce]{chrismbruce1@gmail.com}
\author[Laca]{Marcelo Laca}
\address[Marcelo Laca]{
Department of Mathematics and Statistics\\
University of Victoria\\
Victoria, BC V8W 3R4\\
Canada}%
\email[Laca]{laca@uvic.ca}
%
\author[Ramagge]{Jacqui Ramagge}
\address[Jacqui Ramagge]{
School of Mathematics and Statistics\\
Sydney University\\
NSW 2006\\
Australia}%
\email[Ramagge]{jacqui.rammage@sydney.edu.au}
\author[Sims]{Aidan Sims}
\address[Aidan Sims]{
School of Mathematics and Applied Statistics\\
University of Wollongong\\
NSW 2522\\
Australia}%
\email[Sims]{asims@uow.edu.au}

\thanks{This research was supported by the Natural Sciences
and Engineering Research Council of Canada and Australian Research Council grants
DP150101595 and DP170101821. Part of this work was completed while M.L. and A.S. were
attending the MATRIX@Melbourne Research Program \emph{Refining $C^*$-algebraic invariants
for dynamics using $KK$-theory, July 18--29 2016}.}

\begin{abstract}
Each multiplicative real-valued homomorphism on a quasi-lattice ordered monoid gives rise
to a quasi-periodic dynamics on the associated Toeplitz C*-algebra; here we study the KMS
equilibrium states of the resulting C*-dynamical system. We show that, under a
nondegeneracy assumption on the homomorphism, there is a critical inverse temperature
$\beta_c$ such that at each inverse temperature $\beta \geq \beta_c$ there exists a
unique KMS state. Strictly above $\beta_c$, the KMS states are generalised Gibbs states
with density operators determined by analytic extension to the upper half-plane of the
unitaries implementing the dynamics. These are faithful Type~I states. The critical value
$\beta_c$ is the largest real pole of the partition function of the system and is related
to the clique polynomial and skew-growth function of the monoid, relative to the degree
map given by the logarithm of the multiplicative homomorphism. Motivated by the study of
equilibrium states, we give a proof of the inversion formula for the growth series of a
quasi-lattice ordered monoid in terms of the clique polynomial as in recent work of
Albenque--Nadeau and McMullen for the finitely generated case, and in terms of the
skew-growth series as in recent work of Saito. Specifically, we show that $e^{-\beta_c}$
is the smallest pole of the growth series and thus is the smallest positive real root of
the clique polynomial. We use this to show that equilibrium states in the subcritical
range can only occur at inverse temperatures that correspond to roots of the clique
polynomial in the interval $(e^{-\beta_c},1)$, but we are not aware of any examples in
which such roots exist.
\end{abstract}

\maketitle
\section*{Introduction}
When $C^*$-algebras are used to model systems in quantum statistical mechanics, the time
evolution of the system is modelled as a continuous action of the additive group of real
numbers. Equilibrium states of the system at inverse temperature $\beta$ are modelled by
states satisfying a twisted trace-like condition formalised by Haag, Hugenholtz  and
Winnink \cite{HHW}, and called the \emph{KMS condition} in honour of Kubo, Martin and
Schwinger. Even when a given action $\alpha$ of the reals on a $C^*$-algebra $A$ does not
represent a physical system, the associated simplices of KMS states at various inverse
temperatures turn out to be a very interesting invariant of the pair $(A, \alpha)$. For
example, for the periodic gauge action on a Cuntz--Krieger algebra, they encode the
topological entropy and Perron-Frobenius measure of the underlying shift of finite type
\cite{EFW}; and for the Bost--Connes system they famously encode the Riemann zeta
function as partition function and exhibit a phase transition related to explicit class
field theory \cite{BC,CMR,LLN}. As a result there has been very significant interest in
the KMS structure of $C^*$-dynamical systems in recent years, particularly those related
to combinatorial and algebraic objects such as graphs \cite{CL, Thomsen}, group actions
\cite{LRRW} and semigroups \cite{BaHLR, LR10}.

This paper explores the theme arising in the work of Enomoto--Fujii--Watatani \cite{EFW}
on the relationship between KMS states on Cuntz--Krieger algebras and the entropy of the
underlying shift space. The topological entropy can be regarded as measuring the
asymptotic complexity of the shift space; it measures, roughly speaking, the growth rate,
with respect to $n$, of the number of allowed blocks of length $n$ in the shift space.
Analogously, we might expect that for an appropriate definition of the $C^*$-algebra of a
monoid, and for a dynamics on this $C^*$-algebra encoded by an appropriate length
function on the monoid itself, the resulting KMS data should reflect the growth rate of
the semigroup relative to the given presentation and length.

Here we focus our investigation on the positive cones of quasi-lattice ordered groups.
Recall that $(G,P)$ is quasi-lattice ordered if $G$ is a group with a distinguished
submonoid  $P$ such that in the left order, defined by $x\leq y$ iff $x\inv y \in P$, any
finite collection of elements with a common upper bound has a unique least common upper
bound. The Toeplitz C*-algebra $C^*_\lambda(P)$ is generated by the translation operators
$L_p:\e_x\mapsto \e_{px}$ on $\l^2(P)$. We also consider the universal C*-algebra
$C^*(P)$ for representations of $P$ by isometries that satisfy the Nica-covariance
relations, \cite{N92,LR96}. A multiplicative map  $N : P \to (0,\infty)$ from the monoid
to the positive reals determines a quasi-periodic time evolution on $C^*_\lambda(P)$
given on generators by $\alpha_t(L_p )= N(p)^{it} L_p$, and there is also a corresponding
time evolution in $C^*(P)$. We assume that this homomorphism satisfies conditions that
ensure its logarithm is a degree map in the sense of \cite{S3} and study the KMS states
of the associated $C^*$-dynamical system, establishing a relationship between equilibrium
temperatures and rate of growth of the monoid.

In Section~\ref{sec:main} we describe the semigroups and $C^*$-dynamical systems
$(C^*_\lambda(P), \alpha^N)$ that we study throughout the paper. Following the strategy
laid out in \cite{LR10}, we next establish, in Section~\ref{sec:KMSchar}, an algebraic
characterisation of those states that are KMS$_\beta$ states for a given $\beta$. In
Section~\ref{sec:criticalGibbs} we identify the partition function of the system as a
Dirichlet series~\eqref{tracezeta} whose abscissa of convergence is the critical inverse
temperature $\beta_c$.  For every inverse temperature $\beta > \beta_c$ we construct a
natural generalised Gibbs state and we show that it is the unique KMS$_\beta$-state and
that its GNS representation is a faithful, Type~I$_\infty$ representation. The system
also has a unique KMS$_{\beta_c}$ state and a unique ground state, which is a
KMS$_\infty$ state whenever $\beta_c<\infty$. The critical value $\beta_c$ can be zero,
but only if $P$ is lattice ordered. In Section~\ref{sec:growthrate}, we investigate the
relationship between KMS states and the growth properties of $P$. We show that the
projection $Q_e$ onto the vacuum vector $\e_e$ can be expressed as an operator-valued
Euler-type product, which expands to an operator-valued analogue of the clique polynomial
\cite{AN09,Mc13} and of the skew-growth series \cite{S3}. When we evaluate the
KMS$_\beta$ states for $\beta > \beta_c$ on $Q_e$ through this expansion, and use the KMS
condition, we obtain the inversion formula for the growth series from
\cite[Remark~5.5]{S3}. We finish by showing that, for finitely generated quasi-lattice
orders, any subcritical KMS states must occur at roots of the clique polynomial  and
factor through the boundary quotient of \cite{CL07}. We have not been able to determine
whether any such sporadic states do occur. Indeed, Saito \cite{S3} poses the question of
whether there exists a finitely generated quasi-lattice ordered monoid and a weight
function satisfying his hypotheses such that that the clique polynomial admits any roots
between its smallest root and 1. Our results show that establishing the existence of a
subcritical KMS state would answer this question in the affirmative.

\noindent{\sl Acknowledgment:} This project was initiated during a visit of C.B. to
Wollongong and continued through visits of M.L. to  Sydney and to Wollongong. Both C.B
and M.L. would like to acknowledge this and thank the mathematics departments at
Wollongong and Sydney for their welcoming hospitality.

\section{Quasi-periodic dynamics on the \texorpdfstring{$C^*$}{C*}-algebras of quasi-lattice orders.}
\label{sec:main}
Let  $G$ be a group and suppose $P\subseteq G$ is a submonoid with $P\cap P^{-1}=\{e\}$.
The relation $\leq$ on $G$ defined by $x\le y$ if $x^{-1}y\in P$ is a (left)
translation-invariant partial order on $G$. Following Nica \cite{N92} we say that the
partially ordered group $(G,P)$ is \textit{quasi-lattice ordered} if every finite subset
$F$ of $G$ that has an upper bound has a least upper bound in $G$. The least upper bound
is necessarily unique, and is denoted by $\vee F$, or simply $x\vee y$ if $F = \{x,y\}$.
For such $(G,P)$ we will also say that $\leq$ is a quasi-lattice order on $G$ (or on
$P$), and we often and somewhat loosely refer to $P$ as a quasi-lattice ordered monoid.
We follow \cite{CL07} in that the partial order is defined on all of $G$ and we do not
require $x\vee y$ to be in $P$. When $x$ and $y$ do not have a common upper bound we
extend the notation by saying $x\vee y =\infty$.

\begin{lem}
\label{A1L1} Let $(G,P)$ be a quasi-lattice ordered group and suppose $x,y,z\in G$. Then $x$ and $y$ have a common upper bound  iff $zx$ and $zy$ have a common upper bound,
in which case $zx\vee zy=z(x\vee y)$.
\end{lem}
\begin{proof}
Suppose $x$ and $y$ have a common upper bound; then $zx,zy\leq z(x\vee y)$, so  $zx$ and $zy$ have
a common upper bound. The converse follows on switching from $x,y,z$ to $zx, zy, z\inv$.
 To prove the equality, notice that  $zx,zy\leq zx\vee zy$ implies that $x,y\leq z^{-1}(zx\vee zy)$,
so $x\vee y\leq z^{-1}(zx\vee zy)$. Thus we have $zx,zy\leq z(x\vee y)\leq zx\vee zy$ and since $zx\vee zy$
is the least common upper bound of $zx$ and $zy$
we must have $z(x\vee y)= zx\vee zy$.
\end{proof}

Let $(G,P)$ be a quasi-lattice ordered group and let $\{\e_x : x\in P\}$ be the canonical
orthonormal basis for $\l^2(P)$. For each $p\in P$, the map $L_p:\e_x\mapsto \e_{px}$
extends to an isometry on $\l^2(P)$. This gives a representation $L:P \to \B(\ell^2(P))$
by isometries, called the {\em left regular representation}. The reduced  $C^*$-algebra
$C^*_\lambda(P)$ is the $C^*$-algebra generated by the image of the left regular
representation:
\[
    C_\lambda^*(P):=C^*(\{L_p : p\in P\})\subseteq\B(\l^2(P)).
\]
A fundamental observation made by Nica in \cite{N92} is that, in addition to the obvious
multiplicativity, the isometries in the left regular representation also satisfy the relations
\begin{equation}\label{eq-proj-nica}
L_p L^*_p L_q L^*_q
    = \begin{cases}
        L_{p \vee q} L^*_{p \vee q} &\text{ if $p \vee q < \infty$} \\
        0 &\text{ otherwise,}
    \end{cases}
\end{equation}
which we refer to as Nica-covariance. As a notational convenience we will write $L_\infty
= 0$ so~\eqref{eq-proj-nica} becomes $L_p L^*_p L_q L^*_q = L_{p \vee q} L^*_{p \vee q}$
for all $p,q\in P$.

This observation led Nica to define the (full) $C^*$-algebra of $P$ to be the universal unital $C^*$-algebra
$C^*(P)$ with presentation
\begin{equation}\label{fullpresentation} \left\langle\{v_p\}_{p\in P} :  v_p^*v_p = 1, \  v_pv_q =v_{pq},  \text{ \  } v_pv^*_p v_q v^*_q  = v_{p\vee q}v^*_{p \vee q}, \text{ \  }p,q \in P\right\rangle
\end{equation}
where again $v_{\infty}=0$ by convention. Since the $v_p$ are isometries, multiplying the
relation $v_{p\vee q}v^*_{p \vee q}$ on the left by $v_p^*$ and on the right by $v_q$
gives
\begin{equation}\label{eq-nicacovariance}
    v_p^*v_q = v_{p^{-1}(p\vee q)}v_{q^{-1}(p\vee q)}^* \quad\text{ for all $p,q \in P$.}
\end{equation}
This gives a Wick-ordering of products and implies that the set of products of the form
$v_pv_q^*$ is closed under multiplication and adjoints, hence
\begin{equation}\label{eq:spanning set}
C^*(P) = \overline{\operatorname{span}}\{v_p v^*_q : p,q \in P\}.
\end{equation}
For the details of these constructions, in particular the existence of the universal
$C^*$-algebra and its relationship with the reduced $C^*$-algebra of $P$, see
\cite{N92,LR96}. For more  on the theory of Toeplitz algebras of monoids, see also
\cite{CL02,CL07,Li1,Li2}.

Every multiplicative map $N:P \to (0,\infty)$ on a quasi-lattice ordered monoid gives rise to
a time evolution on $C^*(P)$
for which the generating isometries $v_p$ are eigenvectors.

\begin{prop}\label{dynamicsdefinition}
Let $(G,P)$ be a quasi-lattice ordered group and suppose $N:P \to (0,\infty)$ is a
multiplicative map. Then there exists a strongly continuous one-parameter group
$\{\alpha_t\}_{t\in\RR}$ of automorphisms of $C^*(P)$
 satisfying
\[\alpha_t(v_p)=N(p)^{it}v_p \qquad \text{ for all $p\in P$ and $t\in\RR$.}\]
\end{prop}
\begin{proof}
For each $t\in \RR$,  the collection $\{N(p)^{it}v_p: p\in P\}$ consists  of isometries
that also satisfy the defining relations in the presentation \eqref{fullpresentation}. So
the universal property of $C^*(P)$ gives homomorphisms $\alpha_t: C^*(P) \to C^*(P)$ that satisfy
$\alpha_s\alpha_t = \alpha_{s+t}$ and have  inverses $\alpha_{-t}$, so they are
automorphisms. Continuity of the maps $t\mapsto \sigma_t(a)$ for each $a\in C^*(P)$
follows from a standard ``$\epsilon/3$~argument".
\end{proof}

\begin{remark}
The automorphism group $\alpha_t^N$ from \proref{dynamicsdefinition} descends to a
spatially implemented dynamics on $C^*_\lambda(P)$. Indeed,  for each $t\in \RR$ the map
$U_t:\e_x\mapsto N(x)^{it}\e_x$ for $x\in P$ extends to a unitary operator on
$\ell^2(P)$. A routine argument shows that $\{U_t: t\in \RR\}$ is a strongly continuous
one-parameter unitary group on $\l^2(P)$, and the left regular representation $L: P\to
\mathcal B(\ell^2(P))$ satisfies
\[
 U_t L_p U_t^* = N(p)^{it} L_p  \qquad (p\in P) \ \  (t\in \RR).
\]
We continue to write $\alpha^N$ for this dynamics on $C^*_\lambda(P)$.
\end{remark}

\begin{remark}
If the group $G$ is abelian, the dynamics $\alpha^N$ is a continuous one-parameter
subgroup of the canonical dual action of the compact group $\hat G$ on $C^*(P)$, and is
quasi-periodic in the usual sense of converging averages, see e.g. \cite[Corollary
8]{La98}. In general, when $G$ is nonabelian, there is a coaction of $G$ on $C^*(P)$, see
e.g. \cite[Section 6]{LR96}, and thus also a corresponding quotient coaction of the
(abelian) multiplicative subgroup of $\mathbb{R}^*_+$ generated by $N(P)$.
\end{remark}

\section{Characterisation of KMS states\texorpdfstring{ of $C_\lambda^*(P)$}{}} \label{sec:KMSchar} Let $A$ be a
$C^*$-algebra with a time evolution $\alpha: \RR \to \Aut(A)$. We say that $a \in A$ is
analytic for $\alpha$ if $t \mapsto \alpha_t(a)$ extends to an entire function $z \mapsto
\alpha_z(a)$ from $\CC$ to $A$. Let $\beta\in \RR$. By definition
\cite[Definition~5.3.1]{BR}, a state $\varphi$ of $A$ is KMS$_\beta$ for $\alpha$ if
\begin{equation}\label{KMSdefinition}
\varphi(ab) = \varphi(b \alpha_{i\beta} (a))
\end{equation}
for all $a,b$ in some dense $\alpha$-invariant $^*$-subalgebra of the analytic elements
of $A$. When $(A,\alpha)$ represents a quantum dynamical system, these are the
equilibrium states of $(A,\alpha)$ at inverse temperature $\beta$. If $t \mapsto
\alpha_t(a)$ and $t \mapsto \alpha_t(b)$ have analytic extensions, then $z \mapsto
\alpha_z(a) + \lambda \alpha_z(b)$ is an analytic extension of $t \mapsto \alpha_t(a +
\lambda b)$. So uniqueness of analytic extensions shows that the KMS condition passes
from a given set of analytic elements to its linear span. Hence, to show that a state is
KMS$_\beta$, it suffices to prove that~\eqref{KMSdefinition} holds for every $a$ and $b$
in an $\alpha$-invariant subset of analytic elements of $A$ whose linear span is a dense
$^*$-subalgebra of $A$. The KMS$_0$ states are defined to be the $\alpha$-invariant
traces on $A$.

We are interested in computing KMS states of $C^*(P)$ for the quasi-periodic dynamics
associated to a multiplicative homomorphism $N:P\to(0,\infty)$. Since the function $z
\mapsto (N_p/N_q)^{iz}v_pv_{q}^*$ is a $C^*$-algebra valued entire function for every $p,
q \in P$, the elements $v_pv_{q}^*$ are analytic; and they span an $\alpha$-invariant
dense $^*$-subalgebra of $C^*(P)$ by~\eqref{eq:spanning set}. Thus, we will carry out our
computations on this spanning set.

We begin by showing that if we are interested in studying
KMS states, then the range of $N$ has to lie either entirely within $[1,\infty)$ or entirely within $(0,1]$.
\begin{prop}\label{tempchoice}
Assume there exists a KMS$_\beta$ state of $(C^*(P), \alpha^N)$.
\begin{itemize}
\item[(a)] if $\beta>0$, then  $N(P) \subset [1,\infty)$ and
\item[(b)] if $\beta <0$, then $N(P) \subset (0,1]$.
\end{itemize}\end{prop}
\begin{proof}
Let $\varphi$ be a KMS$_\beta$ state and let $p\in P$. Since $v_p$ is an isometry,
\[
 0 \leq \varphi(1 -  v_pv_{p}^* ) = 1 - \varphi(v_pv_{p}^*) = 1 - N(p)^{-\beta} \varphi(v_{p}^*v_p) = 1 - N(p)^{-\beta}.
\]
Hence $ N(p)^{-\beta} \leq 1$, from which (a) and (b) follow.
\end{proof}

Next we obtain a characterisation of KMS states of $C^*(P)$ in terms of their values on
the spanning elements. To include the case of infinitely generated $P$ we need to assume
that $\inf \{N(p): p \in P\setminus\{ e\}\} >1$; this is automatically satisfied when $P$
is generated by a finite set $\S \subset P\setminus \{e\}$ and $N(s) >1$ for every $s\in
\S$.

\begin{prop}
\label{kmschar} Let $(G,P)$ be a quasi-lattice ordered group,  let $N:P\to [1,\infty)$ be
a multiplicative map such that $\inf\{N(p): p\in P \setminus\{e\}\}
> 1$, and let $\alpha^N$ be the associated dynamics on $C^*(P)$. Suppose that $\beta\in
(0,\infty)$. A state $\varphi$ is a KMS$_\beta$ state for $\alpha^N$ if and only if for every $p,q\in P$,
\begin{equation}\label{KMSformula}
 \varphi(v_pv_{q}^*)=\begin{cases} N(p)^{-\beta} & \text{ if } p=q \\
 0 & \text{ otherwise.}
 \end{cases}
\end{equation}
In particular, for each $\beta > 0$ there is at most one KMS$_\beta$-state for
$\alpha^N$.
\end{prop}

\begin{proof}  Suppose first $\varphi$ is a KMS$_\beta$ state for $\alpha^N$ and let $p,q \in P$.
 When $p=q$ the KMS condition implies $\varphi(v_pv_q^*)=N(p)^{-\beta}\varphi(v_p^*v_p)=N(p)^{-\beta}$.
It remains to prove that $\varphi(v_pv_{q}^*) =0$ when $p \neq q$. Since $p \not= q$, we
can write $\{p,q\} = \{p_0, q_0\}$ with $p_0 \not= e$. Using first that
\begin{equation}\label{TakeAdjoint}
    |\varphi(v_p v^*_q)| = |\overline{\varphi((v_p v^*_q)^*)}| = |\varphi(v_q v^*_p)|,
\end{equation}
and then second the KMS condition, we see that
\begin{equation} \label{KMSCond}
    |\varphi(v_p v^*_q)| = |\varphi(v_{p_0} v^*_{q_0})| = N(p_0)^{-\beta } |\varphi(v_{q_0}^* v_{p_0})|.
\end{equation}

If $q_0 \vee p_0= \infty$, then  Nica-covariance gives $v_{q_0}^* v_{p_0}=0$ and we are
done. If $q_0 \vee p_0< \infty$, then Nica-covariance gives $v_{q_0}^*v_{p_0} =
v_{{q_0}^{-1} (p_0\vee q_0)} v_{p_0^{-1} (p_0 \vee q_0)}^*$. Since $q_0 \not= p_0$,
cancellation in $P$ ensures that $q_0^{-1} (p_0 \vee q_0) \not= p_0^{-1} (p_0 \vee q_0)$,
so again we may write $\{q_0^{-1} (p_0 \vee q_0), p_0^{-1} (p_0 \vee q_0)\} = \{p_1,
q_1\}$ with $p_1 \neq e$. Again using \eqref{TakeAdjoint} with $(p_1,q_1)$ and the
KMS$_\beta$ condition on \eqref{KMSCond} yields
\begin{equation}
    |\varphi(v_pv^*_q)| = N(p_0)^{-\beta} N(p_1)^{-\beta} |\varphi(v_{q_1}^* v_{p_1})|.
\end{equation}

If $q_1 \vee p_1= \infty$, then Nica-covariance gives $v_{q_1}^* v_{p_1}=0$ and we are
done. If $q_1 \vee p_1< \infty$, then we continue the process by induction. If at some
stage  $p_i \vee q_i = \infty$, we obtain $|\varphi(v_pv^*_q)| = 0$ directly. Otherwise
 we generate an infinite sequence $\{(p_n, q_n)\}$  of pairs of elements of $P$ such that $p_n \neq e$ and
 \[
 |\varphi(v_p v^*_q)| = \Big(\prod^n_{i=0} N(p_i)^{-\beta}\Big) |\varphi( v_{p_{n+1}} v_{q_{n+1}}^*)|.
 \]
Since $|\varphi( v_{p_{n+1}} v_{q_{n+1}}^*)| \leq 1$ and $ \prod^n_{i=0} N(p_i)^{-\beta}
\le (\inf_{p\in P\setminus \{e\}}N(p))^{-n\beta} \to 0$ this finishes the proof that
$\varphi(v_p v^*_q)=0$ for $p\neq q$. Hence $\varphi$ satisfies \eqref{KMSformula}.

Next suppose that $\varphi$ is a state on $C^*(P)$ that satisfies \eqref{KMSformula}
 for every $p,q\in P$. In order to show
that $\varphi$ is a KMS$_\beta$ state, it suffices to show that
\[
\varphi(v_{p_1}v_{q_1}^*v_{p_2}v_{q_2}^*)=N(p_1)^{-\beta}N(q_1)^{\beta}\varphi(v_{p_2}v_{q_2}^*v_{p_1}v_{q_1}^*)
\]
for every $p_1,q_1,p_2,q_2\in P$. Fix $p_i, q_i \in P$ for $i =1,2$. By Nica-covariance, when
$v_{p_2}v_{q_2}^*v_{p_1}v_{q_1}^* \not= 0$, it can be written as $v_xv_y^*$ for some $x,y
\in P$. So our assumption on $\varphi$ implies that the right-hand side vanishes unless
$v_{p_2}v_{q_2}^*v_{p_1}v_{q_1}^*$ is self-adjoint, and hence equal to $v_{q_1}v_{p_1}^*v_{q_2}v_{p_2}^*$. Thus,
in order to conclude that $\varphi$ is a KMS$_\beta$ state, it suffices to prove that
\begin{equation}\label{KMSmonomials}
N(p_1)^{\beta}\varphi(v_{p_1}v_{q_1}^*v_{p_2}v_{q_2}^*)=N(q_1)^{\beta}\varphi(v_{q_1}v_{p_1}^*v_{q_2}v_{p_2}^*).
\end{equation}
The Nica-covariance relation for $v_{q_1}^*v_{p_2}$ shows that
\[
v_{p_1}v_{q_1}^*v_{p_2}v_{q_2}^*=\begin{cases}
v_{p_1q_1^{-1}(q_1\vee p_2)}v_{q_2p_2^{-1}(q_1\vee p_2)}^* & \text{if }q_1\vee p_2\neq \infty\\
 0 & \text{otherwise,}
\end{cases}
\]
so~\eqref{KMSformula} gives
\begin{equation}
\varphi(v_{p_1}v_{q_1}^*v_{p_2}v_{q_2}^*)
    = \begin{cases}
        N(p_1q_1^{-1}(q_1\vee p_2))^{-\beta} &\text{if } q_1\vee p_2\neq\infty\text{ and } p_1q_1^{-1}= q_2p_2\inv\\
        0 & \text{otherwise.}
    \end{cases}\label{p1q1p2q2}
\end{equation}
Similarly, Nica-covariance for $v^*_{p_1} v_{q_2}$ followed by~\eqref{KMSformula} shows
that
\begin{equation}
\varphi(v_{q_1}v_{p_1}^*v_{q_2}v_{p_2}^*)
    = \begin{cases}
        N(q_1p_1^{-1}(p_1\vee q_2))^{-\beta} &\text{if } p_1\vee q_2\neq\infty\text{ and } q_1p_1\inv =p_2q_2\inv  \\
        0 & \text{otherwise}.
    \end{cases}\label{p2q2p1q1}
\end{equation}
To see that the two cases in \eqref{p1q1p2q2} match up with those in \eqref{p2q2p1q1}
first note that
\[
p_1q_1^{-1}(q_1\vee p_2) = (p_1\vee p_1q_1^{-1}p_2)= p_1\vee q_2p_2^{-1}p_2= p_1\vee q_2,
\]
where the first equality holds by \lemref{A1L1} and the second one because $p_1q_1^{-1} =
q_2p_2^{-1}$. This  gives
\[
q_1^{-1}(q_1\vee p_2)  = p_1\inv( p_1\vee q_2).
\]
Since $N$ is a homomorphism, we then have
\begin{align*}
 N(p_1)^\beta N(p_1 q_1\inv(q_1\vee p_2))^{-\beta}
  &= N(q_1\inv(q_1\vee  p_2))^{-\beta}\\
  &= N(p_1\inv(p_1\vee  q_2))^{-\beta}\\
  & = N(q_1)^\beta N(q_1 p_1\inv(p_1\vee q_2))^{-\beta}.
\end{align*}
Combining this with \eqref{p1q1p2q2}~and~\eqref{p2q2p1q1} gives~\eqref{KMSmonomials}.
That there is at most one state satisfying~\eqref{KMSformula} follows
from~\eqref{eq:spanning set} together with linearity and continuity of states.
 \end{proof}
\begin{remark}
The same result, with the same proof, holds for the KMS states of $C_\lambda^*(P)$.
From \eqref{KMSformula}, it is clear what the values of a KMS state have to be on a dense subalgebra.
The upshot of
 \proref{kmschar} is that in order to decide whether $C^*(P)$ and $C_\lambda^*(P)$
 admit a KMS state at inverse temperature
$\beta$, we need to decide whether formula \eqref{KMSformula} determines a positive linear functional on $C^*(P)$.\end{remark}

\begin{remark}
We point out that there exist  quasi-lattice ordered groups that do not admit
homomorphisms $N$ with the properties assumed in \proref{kmschar}. This stems from the
fact that such homomorphisms must factor through the abelianisation of the monoid, and
semidirect products provide easy examples without any such homomorphisms. For instance,
any real valued homomorphism $N$ of the affine monoid $\NN\rtimes \NN^\times$ must be
trivial on the additive part $\NN$, so  \proref{kmschar} cannot apply to $C^*(\NN\rtimes
\NN^\times)$. In this case, if one makes the obvious choice $N:\NN\rtimes \NN^\times \to
[1,\infty)$ given by $N(r,a) = a$, then for each $\beta$ above the critical inverse
temperature, the extremal KMS$_\beta$ states are indexed by $\TT = \hat\ZZ$,
\cite[Theorem 7.1]{LR10}.
\end{remark}

\section{Critical temperature and generalised Gibbs states}\label{sec:criticalGibbs}
Recall that the unitaries $\{U_t\}_{t\in \RR}$ determined by $U_t \e_p = N(p)^{it} \e_p$
implement the dynamics $\alpha^N$ spatially on $C^*_\lambda(P)$ and are diagonal with
respect to the standard orthonormal basis $\{\varepsilon_x\} $ of $ \ell^2(P)$. Let $H$
denote the unbounded, diagonal self-adjoint operator on $\ell^2(P)$ with eigenvalues $w_p
= \log N(p) \geq 0$ and corresponding eigenvectors $\e_p$. Applying the analytic
functional calculus, we can write $U_t = \exp(it H)$.

We define $\exp(-\beta H)$ to be the diagonal operator with eigenvalues  $ N(p)^{-\beta}
$ with respect to the standard  basis. Then $\beta \mapsto \exp(-\beta H)$ is a semigroup
of contractions, which can be obtained as the restriction $U_{i\beta}$  to the positive
imaginary axis $\beta\geq 0$ of the operator-valued analytic extension to the upper half
plane of the unitary group $t \mapsto U_t$.
 We emphasise that the operator $\exp(-\beta H)$ is a bona-fide contractive linear operator that
can be defined directly and without any reference to the unbounded operator $H$ because
$N(p)^{-\beta} \leq 1$ for every $p\in P$.

Evaluation of the usual trace on $\B(\ell^2(P))$ at $\exp(-\beta H)$ using the standard
orthonormal basis gives an infinite sum of positive terms,
\begin{equation}\label{tracezeta}
\tr(\exp(-\beta H)) = \sum_{p\in P} N(p)^{-\beta},
\end{equation}
which takes values in $[0,\infty]$ and is decreasing as a function of $\beta\in
[0,\infty]$. We define the {\em critical inverse temperature} to be the abscissa of
convergence,
\begin{equation}\label{****}
\beta_c :=\inf \{ \beta\in (0, \infty): \tr(\exp(-\beta H)) <\infty\},
\end{equation}
with the usual convention that $\inf\emptyset = \infty$.  We collect below the
consequences of the above considerations that will be needed for our analysis of KMS
states, and in particular the KMS$_{\beta_c}$ states when $\beta_c < \infty$. Notice that
$\tr(\exp(-\beta H)) <\infty$ for every $\beta > \beta_c$ and $\sum_{p\in P}
N(p)^{-\beta} = \infty$ for every $\beta < \beta_c$.

\begin{lem} \label{criticalbeta}
Let $N: P \to [1, \infty)$ be a multiplicative morphism such that $N(p) = 1 $ only if $p
= e$ and consider the following statements.
\begin{enumerate}[(a)]
    \item\label{it:betac finite} $\beta_c <\infty$;
    \item\label{it:small Np finite} $\{p\in P: N(p) \leq r\}$ is finite for every $r
        \geq 1$;
    \item\label{it:Nps away from 0} $\inf \{N(p): p \in P, \  p \neq e\} >1$.
\end{enumerate}
Then (\ref{it:betac finite})$\;\implies\;$(\ref{it:small Np
finite})$\;\implies\;$(\ref{it:Nps away from 0}).

Suppose, in addition, that $P$ is finitely generated. Let $\S$ be a finite set of
nontrivial generators and define $\eta:= \min \{N(s): s \in \S\} $. Then $\beta_c \leq
\frac{\log |\S|}{\log \eta} <\infty$.

 \end{lem}
\begin{proof}
If condition~(\ref{it:Nps away from 0}) fails, then there exists a sequence $p_n \in P$
such that $1< N(p_{n+1}) <N(p_n)$, so condition~(\ref{it:small Np finite}) also fails.
Suppose now that~(\ref{it:small Np finite}) fails and fix $x > 1$ such that $\{p\in P:
N(p) \leq x\}$ is infinite. Then the series in \eqref{tracezeta} has infinitely many
terms satisfying $ N(p)^{-\beta} \geq x^{-\beta} $, hence
 diverges for every $\beta \in (0,\infty)$ so condition~(\ref{it:betac finite}) fails.

Assume now $\S \subset P\setminus \{e\}$ is a finite set of generators for $P$ and define
$\eta:= \min \{N(s): s \in \S\} $. Each $p \in P$ can be factored as $p = s_1 s_2 \cdots
s_n$ with $s_k \in \S$. Then $N(p)^{-\beta} \leq \eta^{-\beta n}$, giving
\[
    \sum_{p\in P} N(p)^{-\beta}   \leq \sum_{x\in \FF^+_\S} \eta^{-\beta |x|}.
\]
Since the series on the right converges if and only if $|\S| \eta^{-\beta} <1$, we have
$\beta_c \leq \frac{\log |\S|}{\log \eta}$.
\end{proof}

\begin{remark}
Let $\FF^+_\S$ be  the free monoid generated by a countable set $\S = \{s_1, s_2,
\ldots\}$.
\begin{enumerate}
\item The usual word length gives an example that satisfies~(\ref{it:Nps away from
    0}) but not~(\ref{it:small Np finite}) in \lemref{criticalbeta}.
\item  The homomorphism $N: \FF^+_\S \to [1,\infty)$ determined by $N(s_k) = \log
    (k+2)$ gives an example that satisfies~(\ref{it:small Np finite}) but
    not~(\ref{it:betac finite}) in \lemref{criticalbeta}.
\end{enumerate}
\end{remark}

\begin{prop}\label{dirichletseriesprop}
Suppose $(G,P)$ is a quasi-lattice ordered group and let $N:P \to [1,\infty)$ be a
multiplicative map such that $\{p\in P: N(p) \leq x\}$ is finite for every $x \geq 1$.
Let $(\lambda_n)_{n\in\NN}$ be the strictly increasing listing of the elements of the set
$\{\log N(p): p\in P\}$. Then
\begin{enumerate}
\item $\exp(- H)$ is a positive compact operator on $\ell^2(P)$ with decreasing
    eigenvalue list $1 = e^{-\lambda_0} > \cdots > e^{-\lambda_n} >
    e^{-\lambda_{n+1}} \cdots$, and finite multiplicities $a_n := \#\{p\in P:  N(p) =
    e^{\lambda_n}\}$;
\item $\beta_c = \underset{n\to\infty}{\lim\sup}\;
    \frac{1}{\lambda_n}\log\left(\#\{p: N_p \leq e^{\lambda_n} \}\right ) $
\item if  $\beta_c < \infty$,  the Dirichlet series $ \sum _{n=0}^\infty a_n
    e^{-s\lambda_n} $ is absolutely convergent on the region $\Re s > \beta_c$ to an
    analytic function $s\mapsto \tr(\exp(-s H))$ for which $s = \beta_c$ is a
    singular point.
\end{enumerate}
\end{prop}
\begin{proof}
Since $N$ is multiplicative, $N^{-1}(1)$ is a subsemigroup of $P$, and it is finite by
assumption. Since $P$ has no torsion, it follows that $N^{-1}(1) = \{e\}$. So $a_0 = 1$.

The set $\{ \log N(p) :p\in P\}$ has finite intersection with any bounded segment $[0,
\lambda] $,  so its elements can be listed in increasing order, yielding a sequence
$(\lambda_n) \nearrow \infty$. Part (1)  then follows from the spectral theorem for
self-adjoint compact operators. Since $\#\{p: N_p \leq e^{\lambda_n} \} = \sum_{j=0}^n
a_j$, parts (2)~and~(3) follow from \cite[Theorems 8~and~10]{HR}.
\end{proof}

\begin{definition}
In situations where $\beta_c <\infty$, we define the {\em (generalised) Gibbs state at
inverse temperature $\beta >\beta_c$} to be the state on $\B(\ell^2(P))$ defined by
\[
\psi_\beta(X)  := \tr(\exp(-\beta H))\inv \tr(X \exp(-\beta H)) \qquad X \in \B(\ell^2(P)).
\]
This gives a state on $C_\lambda^*(P)$. Let $\lambda : C^*(P) \to C^*_\lambda(P)$ be the
left-regular representation $\lambda(v_p) = L_p$. Then $\psi_\beta \circ \lambda$ is a
state on $C^*(P)$, which we again call the generalised Gibbs state on $C^*(P)$ at inverse
temperature $\beta$.
\end{definition}

For the next result, recall that a {\em ground state} for $\alpha$ is a state $\phi$ such
that the entire function \begin{equation}\label{ground} z \mapsto \phi(y \alpha_z(x)) ,
\qquad z\in \CC,
\end{equation}
is bounded on the upper half plane for every  $x,y \in C^*_\lambda(P)$ with $x$ analytic.

\begin{thm} \label{phasetransitionmain}
Let $(G,P)$ be a quasi-lattice ordered group and suppose $N:P\to [1,\infty)$ is a
multiplicative map such that $N(p) = 1$ only if $p = e$, let $\alpha^N$ be the associated
dynamics on $C_\lambda^*(P)$, and assume $0<\beta_c < \infty$. Then
\begin{enumerate}
\item for each $\beta > \beta_c$, the Gibbs state $\psi_\beta$ is the unique
    KMS$_\beta$ state of the dynamical system $(C_\lambda^*(P), \alpha^N)$ and its
    GNS representation is a faithful, type~I$_\infty$ representation of
    $C_\lambda^*(P)$;
\item  if $Q_e$ denotes the projection onto the vacuum subspace $\CC \e_e$ of
    $\ell^2(P)$, then
\[
    \tr(\exp(-\beta H))\  \psi_\beta(Q_e)  = 1 \qquad \text (\beta > \beta_c);\quad\text{ and}
\]
\item there is a unique KMS$_{\beta_c}$ state for $(C_\lambda^*(P), \alpha^N)$.
\end{enumerate}
Whether or not $\beta_c < \infty$, the Fock state $\phi_e(\cdot) = \langle \pi_L(\cdot)
\e_e, \e_e\rangle$ corresponding to the vacuum vector $\e_e$ is the unique ground state
for $(C_\lambda^*(P), \alpha^N)$. This ground state is a KMS$_\infty$ state if
$\beta_c <\infty$.
\end{thm}
\begin{proof}
Assume first $\beta >\beta_c$, so that  $Z_N(\beta) := \tr(\exp(-\beta H))<\infty$. By
definition of $\psi_\beta$, we have
\begin{equation*}
\psi_\beta(v_p v_{q}^*):=\frac{1}{Z_N(\beta)}\sum_{x\in P}N(x)^{-\beta}\inn{L_p L_{q}^*\e_x}{\e_x}.
\end{equation*}
Since \[\inn{L_p L_{q}^*\e_x}{\e_x} = \inn{L_q^*\e_x}{L_p^*\e_x}=\begin{cases}
1 & \text{ if } p=q\leq x \text{ and }\\
0 & \text{otherwise,}
\end{cases}\]
we have,
\begin{equation}\label{computationabove}
\begin{split}
\psi_\beta(v_pv_q^*)
    &= \frac{1}{Z_N(\beta)}\sum_{x\in P}N(x)^{-\beta}\inn{L_q^*\e_x}{L_p^*\e_x}\\
    &=\delta_{p,q}\frac{1}{Z_N(\beta)}\sum_{w\in P}N(pw)^{-\beta }=\delta_{p,q} N(p)^{-\beta}.
\end{split}
\end{equation}
Hence $\psi_\beta$ satisfies \eqref{KMSformula}, and is therefore the unique KMS$_\beta$
state by \proref{kmschar}. It factors through $C_\lambda^*(P)$ by construction.

Let $L$ be the left regular representation of $P$ on $\ell^2(P)$. Then
\begin{equation}\label{projectionQ}
Q_e := \prod_{p\in P\setminus\{ e\}} (1 - L_p L_p^*)
\end{equation}
is the rank-one projection onto the basis vector in $\ell^2(P)$ corresponding to the
identity $e\in P$. The infinite product~\eqref{projectionQ} is the limit of a decreasing
net of projections indexed by the finite subsets of $P\setminus \{e\}$ and is in the von
Neumann algebra $C^*_\lambda(P)''$. For $p,q \in P$ the rank-one operator
$\inn{\cdot}{\varepsilon_q}\varepsilon_p$ is the product $L_p Q_e L_q^*$, which is also
in $C^*_\lambda(P)''$. Thus the compact operators are contained in $C^*_\lambda(P)''$ and
the left regular representation $\pi_L$ is irreducible.

The assertion about faithfulness and type holds because the range of the density operator
defining $\psi_\beta$ is generating for $\pi_L$, so the GNS representation of
$\psi_\beta$ is quasi-equivalent to $\pi_L$, see e.g. \cite[Lemma 3.2]{La93}.

A computation analogous to \eqref{computationabove} shows that $\psi_\beta(L_p Q_e L_p^*)
= N_p^{-\beta}\psi_\beta(Q_e)$ and since the series $\sum_p L_p Q_eL_p^*$ converges
strongly and monotonically to the identity and $\psi_\beta $ is a normal state on
$\B(\ell^2(P))$, we have
\[
1 = \psi_\beta(1) = \sum_p \psi_\beta(L_p Q_eL_p^*) = \sum_p N_p^{-\beta} \psi_\beta(Q_e) = Z_N(\beta) \  \psi_\beta(Q_e).
\]
To see that a KMS$_\beta$ state exists at $\beta = \beta_c <\infty$, recall that the
states of the unital $C^*$-algebra $C^*(P)$ form a weak* compact set so there exists a
weak* convergent sequence $\psi_{\beta_n}$ with $\beta_n \to \beta_c^+$, which by
\cite[Proposition~5.3.25]{BR} converges to a KMS$_{\beta_c}$ state $\psi_{\beta_c}$.  The uniqueness
follows from \proref{kmschar}.

To prove the final two assertions, suppose that $\phi$ is a ground state. We have
\[
\phi(L_p \alpha_{i\beta} (L_q^*)) = N_q^{\beta}\phi(L_p L_q^*),
\]
so $\phi$ vanishes on $L_p L_q^*$ if $q \neq e$, and hence (by taking adjoints) also if
$p\neq e$. Thus $\phi$ is the  Fock state. Conversely, direct computation shows that the
Fock state is a ground state. If $\beta_c < \infty$, we may take  limits as $\beta
\to\infty$ in \eqref{computationabove}, which shows that the Fock state is a KMS$_\infty$
state.		
\end{proof}

\begin{remark}
What happens below the critical temperature, when it is strictly positive, is somewhat
mysterious. But we will see later in \proref{subcriticalKMS} that for most values $\beta
<\beta_c$ there are no KMS$_\beta$ states.
\end{remark}

The case $\beta_c =0$ requires special consideration. Recall first that, by definition, a
KMS$_0$ state (or {\em chaotic equilibrium state}) is an $\alpha^N$-invariant trace.

\begin{prop}\label{prp:crit=0}
Let $(G,P)$ be a quasi-lattice ordered group and suppose $N:P\to [1,\infty)$ is a
multiplicative map such that $N(p) = 1$ only if $p = e$. Consider the following
conditions:
\begin{enumerate}
\item $\beta_c =0$;
\item the $C^*$-algebra $C^*(P)$ has a KMS$_0$ state that vanishes at $v_pv_q^*$ for
    $p\neq q$;
\item the $C^*$-algebra $C^*(P)$ has a tracial state;
\item  $P$ is lattice ordered.
\end{enumerate}
Then (1)~$\implies$~(2)~$\iff$~(3)~$\iff$~(4)
\end{prop}
\begin{proof}
Suppose first that $\beta_c =0$. Applying Theorem~\ref{phasetransitionmain}(1) to each
$\beta \in \{1/n : n = 1, 2, \dots\}$ gives a sequence $\psi_n$ of KMS$_{1/n}$ states
that all vanish on $v_p v^*_q$ for $p \not= q$. These are automatically
$\alpha^N$-invariant because each $1/n > 0$. Taking the limit of any weak$^*$-convergent
subsequence of the $\psi_n$ gives~(2) by \cite[Proposition~5.3.25]{BR}. This proves
(1)~implies~(2). That (2)~implies~(3) is straightforward.

Suppose now $\tau$ is a tracial state on $C^*(P)$ and let $p, q\in P$. Then
\[
 \tau(v_pv_{p}^*+ v_qv_{q}^*) = \tau(v_pv_{p}^*)+  \tau(v_qv_{q}^*) = \tau(v_{p}^*v_p)  + \tau(v_{q}^*v_q) =2,
 \]
so  $v_pv_{p}^* +v_qv_{q}^*$ cannot be a projection, which means that  $v_pv_{p}^*$ and
$v_qv_{q}^*$ are not orthogonal.  Hence $p\vee q <\infty$ for every $p,q\in P$ and $P$ is
lattice ordered. This proves (3)~implies~(4).

To prove (4)~implies~(2), assume $P$ is lattice ordered. Then the restriction to $P$ of
any unitary representation of $G$ trivially satisfies Nica-covariance. So the universal
property of $C^*(P)$ gives a homomorphism $C^*(P) \to C^*(G)$ that carries generators to
generators. Post-composing this map with the canonical trace on $C^*(G)$ yields a trace
on $C^*(P)$.
\end{proof}

\begin{remark}
In the proof of \proref{prp:crit=0}, we cannot conclude that $C^*(P)$ admits a unique
$\alpha^N$-invariant trace because \proref{kmschar} does not apply at $\beta =0$. To see
what goes wrong, consider $P = \NN^2$ with its usual order and $N : \NN^2 \to \RR$ given
by $N(n) = n_1 + n_2$. This yields the dynamics on $C^*(\ZZ^2)$ given by $\alpha_t(U_m) =
e^{it(m_1 + m_2)}$, which fixes the subalgebra $\clsp\{U_{(j,-j)} : j \in \ZZ\} \cong
C(\TT)$. So every probability measure on $\TT$ determines an $\alpha$-invariant trace on
$C^*(P)$.
\end{remark}

\begin{remark}
For each integer $n \geq 2$ the free monoid $P = \FF_n^+$ on $n$ generators gives rise to
$C^*(\FF_n^+)\cong \T\O_n$, the Toeplitz extension of the Cuntz algebra $\O_n$. The
dynamics corresponding to the choice $N_s = 1$ for each generator $s$ of $\FF_n^+$ is the
usual periodic gauge action. In this case  we have $\beta_c = \log n$ \cite{ole-ped,eva}.
Consideration of the free monoid on infinitely many generators leads to $\T\O_\infty$,
which has no KMS$_\beta$ states at any finite $\beta$. On the other hand, taking, for
example, $N(i) = 2^i$ for each generator $i$ of $\mathbb{F}^+_\infty$ gives a dynamics on
$\mathcal{TO}_\infty$ that admits KMS$_\beta$ states at many finite inverse temperatures.
\end{remark}

\section{Growth rate and inversion formula}\label{sec:growthrate}
The partition function $Z_N(\beta) :=
\tr(\exp(-\beta H)) = \sum_{p\in P} N_p^{-\beta}$ of the system $(C^*_\lambda(P), \alpha^N)$
 and in particular its abscissa of convergence
are intrinsically related to the growth of $P$ relative to $N$. Indeed, $\beta \mapsto
Z_N(\beta) = \tr(\exp(-\beta H))$ is the growth series for $P$ relative to the weight $w
= \log N$ evaluated at $t= e^{-\beta}$, \cite{Mc13,AN09,S1,S2,S3}. We explore next this
interesting connection of the classification of equilibrium states and recent work on
monoid growth.

Recall from the proof of \thmref{phasetransitionmain} that $Q_e = \prod_{p\in
P}(1-L_pL_p^*)$, where the infinite product is viewed as the monotone limit of a
decreasing net of projections indexed by the finite subsets of $F$ of $P\setminus \{e\}$.
We explore the relationship between $Q_e$ and the skew growth series and the clique
polynomial of $P$, cf. \cite{AN09,Mc13,S3}. Recall that a {\em clique} in a quasi-lattice
ordered monoid $P$ is a finite subset $F$ that has a least upper bound $ \vee F$ in $P$.
The motivating example is that of a right angled Artin monoid, generated by the vertices
of a simplicial graph $\Gamma$, in which a pair of generators commutes if and only if
they are joined by an edge. In this case the cliques are the subsets of generators that
correspond to finite full subgraphs of $\Gamma$. By convention, we shall admit the empty
set as a clique and write $\vee \emptyset = e$.
 For any subset $A \subseteq P$, we define $\cl(A) := \{F \subseteq A : F\text{ is a
clique}\}$.

Let $P$ be a quasi-lattice ordered monoid with a homomorphism $N: P\to [1,\infty)$ as in \thmref{phasetransitionmain}.
Then  $w = \log N$ is a degree map as defined in \cite[Section 4]{S3} and $P$ satisfies the descending chain condition.
It follows that the set $\S$ of minimal elements in $P$  generates $P$,
in fact $\S$ is the smallest generating subset. When $\S$ is finite, the \emph{clique polynomial} is
 \[
 C_{\S,w}(t) :=  \sum_{F \in \cl(\S)} (-1)^{|F|} t^{w(\vee F) },
 \]
see \cite{CF69,AN09,Mc13}. The \emph{skew-growth series} defined by  Saito in \cite{S3}
provides a vast generalization for cancelative monoids endowed with a degree map, and
consists of alternating sums over towers. Since in the present situation $P$ has the
least upper bound property, Saito's towers have height at most 1, and are indexed by
cliques themselves, see \cite[Example 2]{S3}. Thus, when the  set   $\S$ of minimal
elements is infinite,  the skew-growth function $N_ {P,\operatorname{deg}}$, relative to
the degree map $\operatorname{deg} = w$, is given by the same formula, now interpreted as
a formal infinite \emph{clique series}, which for  $t=e^{-\beta}$ becomes
\begin{equation}\label{eq:clique poly}
 C_{\S,w}(e^{-\beta}) = \sum_{F \in \cl(\S)} (-1)^{|F|} e^{-\beta w(\vee F) } .
\end{equation}
We show that the vacuum projection $Q_e$ has an expression as an operator-valued product
over generators that is analogous to (the inverse of) the familiar Euler product over the
prime numbers. This product has an obvious expansion for finite  $\S$, namely
 \begin{equation}\label{eq:operator clique series}
   \sum_{F \in \cl(\S)} (-1)^{|F|} L_{\vee F} L_{\vee F}^*,
\end{equation}
which is a finite sum and belongs to $ C^*_\lambda(P)$. When  $\S$ is infinite,
\eqref{eq:operator clique series} is an infinite sum, which we interpret as
 the strong operator limit of the decreasing net $\{\sum_{F \in \cl(\F)} (-1)^{|F|} L_{\vee F} L_{\vee F}^*\}_{\F}$,
 indexed by the finite subsets $\F$ of $\S$ directed by inclusion, and is an element of $C^*_\lambda(P)''$.

\begin{prop}
Suppose $(G,P)$ is a quasi-lattice ordered group, and let $\S$ be the set of minimal elements of $P\setminus \{e\}$. Then
\begin{equation}\label{projectioncliquepoly}
 Q_e = \prod_{s\in \S} (1-L_sL_s^*) =\sum_{F \in \cl(\S)} (-1)^{|F|} L_{\vee F} L_{\vee F}^* .
 \end{equation}
\end{prop}
\begin{proof}
Since $\S \subset P\setminus \{e\}$ it is clear that $Q_e \leq  \prod_{s\in \S}
(1-L_sL_s^*)$. For every $p \in P\setminus \{e\}$ there exists $s\in \S$ such that  $p = sp'$
for some $p'\in P$, so $L_p L_p^* = L_s L_{p'} L_{p'}^*L_s^* \leq L_sL_s^*$,
and thus $L_pL_p^* (1 - L_sL_s^*) =0$. Hence
$(1-L_pL_p^*) (1-L_sL_s^*) = 1-L_sL_s^*$, from which we see that $Q_e \geq  \prod_{s\in
\S} (1-L_sL_s^*)$. This proves that $Q_e = \prod_{s\in \S} (1-L_sL_s^*)$, as a strong
limit when $\S$ is infinite.

Suppose now that $\F$ is a finite subset of $\S$. Then
\[
\prod_{s\in \F} (1-L_sL_s^*) = \sum_{F\subset \F} (-1)^{|F|} \prod_{s\in F} L_s L_s^*.
\]
By Nica-covariance, $\prod_{s\in F} L_sL_s^* $ is equal to $ L_{\vee F} L_{\vee F}^* $ or
to $0$, according to whether $F$ is a clique or not. Since $ \prod_{s\in \S} (1-L_sL_s^*)
= \lim_{\F \nearrow \S} \prod_{s\in \F} (1-L_sL_s^*)$, this finishes the proof.
\end{proof}

\begin{cor}\label{cor:inversion}
If $\beta > \beta_c$, then we have the inversion formula
\[
  \Big(\sum_{p \in P} N_p^{-\beta}\Big)  \Big( \sum_{F \in \cl(\S)} (-1)^{|F|} N_{\vee F}^{-\beta }  \Big) =1.
\]
\end{cor}
\begin{proof}
For each $\beta > \beta_c$,  the generalized Gibbs state $\psi_\beta$ on $\B(\ell^2(P))$
is normal in the left regular representation, and since
the right hand side of \eqref{projectioncliquepoly} is the strong limit of
a bounded net indexed by finite sets, we obtain
\[
    \psi_\beta(Q_e) =  \sum_{F \in \cl(\S)} (-1)^{|F|} N_{\vee F}^{-\beta },
\]
where the right hand side is the limit of sums over
cliques of finite subsets of $\S$, and
is a conditionally convergent numerical
series.
The claim now follows from
\thmref{phasetransitionmain}(2).
\end{proof}
\begin{remark}
The convergence of the clique series in \corref{cor:inversion} is conditional and depends on our convention of
adding in stages over the cliques of finite subsets of $\S$. One could list $\S$, for instance so that $w(s_n)$ is nondecreasing,
and obtain a proper series of partial sums given by initial segments of this list,
but convergence would remain conditional. In general,
the terms cannot  be rearranged to write the clique series with $t = e^{-\beta}$
as a Dirichlet series, but the inversion formula itself shows that the
limit function has a meromorphic extension to the half plane $\Re(s) >\beta_c$,
with poles at the zeros of $Z_N(s)$.
\end{remark}
\begin{remark}
If the clique series converges absolutely on a region $\Re s > \sigma_0$ for some $\sigma_0 <\infty$,
then  \corref{cor:inversion} gives the analytic inversion formula (****) in \cite[Remark 5.5]{S3},
\[
\sum_{p \in P} e^{-s w(p)}  \ C_{\S,w}(e^{-s}) =1, \quad \Re(s) > \max\{\beta_c,\sigma_0\}
\]
in the particular case of $P$ quasi-lattice ordered  with degree function $w = \log N$ and minimal set $\S$.
\end{remark}
For finitely generated $P$  the skew growth series is the clique polynomial,
and we have the following result.
\begin{prop}[cf. Theorem 4.2 of \cite{Mc13} and Example 1 of \cite{S3}] \label{subcriticalKMS}
Resume the hypotheses of \thmref{phasetransitionmain}, and suppose further that
$P$ is finitely generated.
Then
\begin{enumerate}
\item  $e^{-\beta_c}$ is the smallest root of the clique polynomial $ C_{\S,w}(t) $ and  the KMS$_{\beta_c}$ state $\varphi_{\beta_c}$
 vanishes on the compact operators $\K(\l^2(P))$;

\item if $\varphi$ is a KMS$_\beta$ state for some $\beta < \beta_c$,
then $e^{-\beta}$ is a root of $C_{\S,w}$ and $\varphi$  vanishes on the compact operators
$\K(\l^2(P))$.
 \end{enumerate}
\end{prop}
\begin{proof}
If $C_{\S,w}(e^{-\beta_c}) \neq 0$, then the inversion formula of \corref{cor:inversion}
shows that $Z_N(s)$ has an analytic extension to a neighborhood of $\beta_c$, contradicting
the last assertion of \proref{dirichletseriesprop}.  This shows that $Z_N(s)$ has a pole at $\beta_c$.

 For each $p\in P$, we know $V_pQ_eV_p^*=\theta_{p,p}\in\K(\l^2(P))$, so $I-\sum_{p\in
F}V_pQ_eV_p^*\geq 0$ for each finite subset $F\subseteq P$. If $\varphi$ is a KMS$_\beta$ state, then
\begin{equation}
\label{SeriesInequality}
1 \geq \sum_{p\in F}\varphi(V_pQ_eV_p^*)=\sum_{p\in F}e^{-\beta w(p)}\varphi(Q_e)=\sum_{p\in F}e^{-\beta w(p)}C_{\S,w}(e^{-\beta}).
\end{equation}
For $\beta < \beta_c$, the sum $\sum_{p\in P}e^{-\beta w(p)}$ diverges, and hence
\eqref{SeriesInequality} implies that $e^{-\beta}$ is  a root of the polynomial
$C_{\S,w}(t)$. Hence there are no KMS$_\beta$ states for any $\beta<\beta_c$ such that
$C_{\S,w}(e^{-\beta})\neq 0$.

Now suppose that $e^{-\beta}$ is a root of the clique polynomial and that $\varphi$ is a
KMS$_\beta$ state for $\alpha^N$. Then $\varphi(Q_e) = C_{\S,w}(e^{-\beta})=0$.
Since $Q_e \perp L_p L^*_p$ for all $p \not= e$,
\eqref{eq:spanning set} shows that the ideal $\langle Q_e \rangle$ generated by
$Q_e$ is equal to $\overline{\text{span}}\{V_pQ_eV_q^*: p,q \in P\}$. For $p,q \in P$, the KMS
condition and the Cauchy--Schwartz inequality show that
\[
|\varphi(V_pQ_eV_q^*)|^2
    = e^{-2\beta w(p)}|\varphi(V_q^*V_pQ_e)|^2
    \leq e^{-2\beta w(p)}|\varphi(V_q^*V_pV_p^*V_q)\varphi(Q_e)|
    =0,
\]
so $\varphi$ vanishes on the ideal $\langle Q_e\rangle$. An elementary calculation shows
that for $p,q \in P$ the product $L_p Q_e L^*_q $ is the rank-1 operator $\theta_{p,q}$,
so $\langle Q_e \rangle = \mathcal{K}(\ell^2(P))$.

Since $Z(s)$ has no poles on $\Re(s) > \beta_c$, the function $C_{\S,w}(e^{-s})$ has no
zeros there, hence no  root of $C_{\S,w}$ can have an absolute value smaller than
$e^{-\beta_c}$.
\end{proof}

\begin{remark}
We expect a similar result for infinitely generated quasi-lattice monoids, with the skew growth
function in place of the clique polynomial. However, the situation here is more delicate:
such a result would depend on analytic or meromorphic continuation of the skew growth function
up to the abscissa of convergence of $Z(\beta)$.
\end{remark}

\end{document}